\documentclass[11pt]{amsart}
\usepackage{amssymb}
\usepackage{palatino}
\usepackage{tikz-cd}
\input amssym.def

\usepackage{xcolor}
\usepackage{amsmath, amsfonts}
\usepackage{amssymb}
\usepackage{amscd}
\usepackage[mathscr]{eucal}
\usepackage{palatino}
\usepackage{theoremref}
\usepackage{tikz}
\usepackage{float}

\newfont{\cyrr}{wncyr10}

\newcommand{\thmref}[1]{Theorem~\ref{#1}}

\newcommand{\lemref}[1]{Lemma~\ref{#1}}

\newtheorem{thm}{Theorem}

\newtheorem{lem}[thm]{Lemma}

\newtheorem{rmk}{Remark}[section]

\newcommand{\Z}{{\mathbb Z}}

\makeatletter
\@namedef{subjclassname@2020}{2020 Mathematics Subject Classification}
\makeatother

\def\({\left(}
\def\){\right)}
\def\[{\left[}
\def\]{\right]}

\def\N{\mathbb{N}}
\def\R{\mathbb{R}}
\def\C{\mathbb{C}}

\title{On Hecke eigenvalues of Ikeda lifts}

\author{Sanoli Gun and Sunil Naik}

\address{Sanoli Gun  \newline
The Institute of Mathematical Sciences, 
A CI of Homi Bhabha National Institute, 
CIT Campus, Taramani, 
Chennai 600 113, 
India.}
\email{sanoli@imsc.res.in}

\address{Sunil L Naik \newline
Department of Mathematics,
Queen's University, Jeffrey Hall, 
99 University Avenue, 
Kingston, ON K7L3N6, 
Canada}
\email{naik.s@queensu.ca}

\begin{document}
	
\hfuzz 5pt	

\subjclass[2020]{11F11, 11F30, 11F46, 11N56, 11B65}

\keywords{Ikeda lifts, Non-vanishing and bounds of Hecke eigenvalues,  
$q$-binomial theorem, Reciprocal polynomials}	
	
\maketitle
		
\begin{abstract}
A well known result of Breulmann states that
Hecke eigenvalues of Saito-Kurokawa lifts are positive.
In this article, we show that the Hecke eigenvalues of an
Ikeda lift at primes are positive. Further, we derive lower 
and upper bounds of these Hecke eigenvalues for all primes $p$. 
One of the main ingredients involves expressing the Hecke 
eigenvalues of an Ikeda lift in terms of certain reciprocal polynomials. 
\end{abstract}
	
\section{Introduction and Statements of Results}
Throughout the article, let $k, n$ denote 
even positive integers with $k > n+1$ and
$\Gamma_n = Sp_n(\Z) \subseteq {\rm GL}_{2n(\Z)}$ 
denotes the full Siegel modular group of degree $n$. 
Also let $S_k(\Gamma_n)$ denotes the space of Siegel cusp forms 
of weight $k$ and degree $n$ for $\Gamma_n$. 
Let $F \in S_k(\Gamma_n)$ be an Ikeda lift with 
Hecke eigenvalues $\{\lambda_F(m) : m \in \N\}$. 
When $n=2$, $F$ is called a Saito-Kurokawa lift.

In 1999, Breulmann \cite{SB} proved that 
if $F$ is a Saito-Kurokawa lift, then 
$\lambda_F(m) > 0$ for every $m \in \N$. 
In \cite{GPS}, Gun, Paul and Sengupta  
derived a lower and upper bound for Hecke eigenvalues 
of Saito-Kurokawa lifts $F$. 
More precisely, they showed that there exist 
positive absolute constants $c_1$ and $c_2$ such that
$$
m^{k-1} \exp\(-c_2 \sqrt{\frac{\log m}{\log\log m}}\) 
~\leq~ \lambda_F(m) ~\leq~
m^{k-1} \exp\(c_1 \sqrt{\frac{\log m}{\log\log m}}\) 
$$
for all $m \geq 3$.

In this article, we prove that Hecke eigenvalues of Ikeda lifts 
at primes are positive. Further, we derive lower and upper bounds 
for Hecke eigenvalues of Ikeda lifts at primes.
More precisely, we prove the following theorem.
\begin{thm}\label{thm4}
Let $F \in S_k(\Gamma_n)$ be an Ikeda lift with 
Hecke eigenvalues $\{\lambda_F(m) : m \in \N\}$. Then we have
$$
\lambda_F(p) ~>~ 0
$$
for all primes $p$. Further, we have
$$	
p^{\frac{nk}{2} - \frac{n(n+1)}{4} +\frac{n^2}{8}}
\prod_{i=1}^{\frac{n}{2}} \(1- \frac{1}{p^{i-\frac{1}{2}}}\)^2
~\leq~ \lambda_F(p) ~\leq~
p^{\frac{nk}{2} - \frac{n(n+1)}{4} +\frac{n^2}{8}}
\prod_{i=1}^{\frac{n}{2}} \(1+ \frac{1}{p^{i-\frac{1}{2}}}\)^2
$$ 
for all primes $p$.
\end{thm} 

\begin{rmk}
In a recently posted article \cite{Ad} on arxiv, it has been claimed
that $\lambda_F(p)$ is positive if $p$ is sufficiently large.
The method used in this paper is different than the one used in \cite{Ad}.
We express Hecke eigenvalues in terms of reciprocal polynomials,
and then use the properties of the reciprocal polynomials and
$q$-binomial theorem to derive upper and lower bounds of $\lambda_F(p)$
for all primes $p$ and hence non-negativity of $\lambda_F(p)$
for all primes $p$.
\end{rmk}

\section{Prerequisites}

\subsection{Prerequisites from $q$-binomial coefficients}\label{subsec q-binom}
 Let $q \ne 1$ be a complex number. For an integer $n$, the 
 $q$-analogue of $n$ is defined by
$$
(n)_q ~=~ \frac{q^n-1}{q-1}.
$$	
Note that for any integer $n \geq 1$, $(n)_q = 1 + q +q^2 + \cdots +  q^{n-1} \in \Z[q]$. 
Further, note that
$$
\lim_{q \to 1} (n)_q ~=~ \lim_{q \to 1} \frac{q^n-1}{q-1} ~=~ n.
$$
For any positive integer $n$, the $q$-factorial $(n)_q!$ is defined by
$$
(n)_q! ~=~ (n)_q(n-1)_q \cdots (1)_q
$$
and we set $(0)_q!=1$. For non-negative integers $m, n$ with $m \leq n$, 
the $q$-binomial coefficient is defined by
$$
\binom{n}{m}_q ~=~ \frac{(n)_q! }{(m)_q! (n-m)_q! }.
$$
Note that $\binom{n}{m}_q \in \Z[q]$ (see \cite[Theorem 2.1]{Co}, 
\cite[p. 17]{Cau}) and in particular, we have $\binom{n}{m}_q \in \Z$ 
if $q$ is an integer. 
We have the following $q$-analogue of the binomial 
theorem (see \cite[Theorem 2.2]{Co}, \cite[p. 46]{Cau}).

\begin{thm}\label{thmqbinom}
Let $n \geq 1$ be an integer. Then we have
$$
\prod_{i=0}^{n-1}(1+q^i x) 
~=~ 
\sum_{j=0}^{n} \binom{n}{j}_q q^{\frac{j(j-1)}{2}} x^j.
$$
\end{thm}

\medspace

\subsection{Prerequisites from Siegel modular forms}
As before, let $S_k(\Gamma_n)$ be the space of Siegel cusp forms 
of weight $k$ and degree $n$. Please see \cite{AAb, BGHZ, EZ, HM} 
for an introduction to Siegel modular forms.

\subsubsection{Saito-Kurokawa lift}
Let $f \in S_{2k-2}(\Gamma_1)$ be a normalized elliptic Hecke eigenform
with Fourier coefficients $\{ a_f(m) \}_{m \ge 1}$. 
It was conjectured by Saito and Kurokawa \cite{NK} that 
there exists a Hecke eigenform $F \in S_k(\Gamma_2)$ such that
$$
Z_F(s) ~=~ \zeta(s-k+1) \zeta(s-k+2) L(s, f).
$$
Here $Z_F(s)$ is the spinor zeta function associated with $F$ defined by
$$
Z_F(s) ~=~ \zeta(2s-2k+4) \sum_{m=1}^{\infty} \frac{\lambda_F(m)}{m^s}
$$
and $L(s, f) = \sum_{m=1}^{\infty} \frac{a_f(m)}{m^s}$ is the modular
$L$-function associated with $f$.
 This conjecture was resolved by Maass, Andrianov and Zagier 
(see \cite{AAc, HM1, HM2, HM3, DZ}). 

We have the following expression (see \cite{SB}, \cite[p. 4]{GPS}) relating 
Hecke eigenvalues of $F$ and Fourier coefficients of $f$ at primes $p$ :
$$
\lambda_F(p) ~=~ a_f(p) + p^{k-2} + p^{k-1}.	
$$
	
\subsubsection{Ikeda lift} 
As before, let us assume that $k> n+1$. A generalization of the Saito-Kurokawa lift 
to higher degrees was predicted by Duke and Imamo\={g}lu.  
They conjectured that for a normalized elliptic Hecke eigenform 
$f \in S_{2k-n}(\Gamma_1)$,  
there exists a Hecke eigenform $F \in S_k(\Gamma_n)$ such that 
$$
L(s, F ; st) ~=~ \zeta(s) \prod_{i=1}^{n} L(s+k-i, f).
$$
Here $L(s, F ; st)$ denotes the  standard L-function associated to $F$ 
(see \cite[p. 221]{BGHZ}). The existence of such a lift was proved by Ikeda \cite{TI} 
and this lift $F$ is now called an Ikeda lift of $f$. 
Note that when $n=2$, it coincides with Saito-Kurokawa lift. 
For any prime $p$, the explicit relation between 
the $p$-th Hecke eigenvalues of $f$ and its Ikeda lift $F$ (see \cite[Eq. 1]{RK}) is given by
\begin{equation}\label{eqlambdap}
\begin{split}
\lambda_F(p) 
~=~ 
p^{\frac{nk}{2} - \frac{n(n+1)}{4}} 
\left( 
\sum_{j=1}^{\frac{n}{2}} 
\sum_{r=0}^{\lfloor \frac{j}{2} \rfloor}
(-1)^r \frac{j}{j-r} 
\binom{j-r}{r} \binom{n}{\frac{n}{2}-j}_p p^{c_{j, r}}  
a_f(p)^{j-2r} 
~+~ 
p^{-\frac{n^2}{8}} \binom{n}{\frac{n}{2}}_p 
\right),
\end{split}
\end{equation}
where 	
$$
c_{j,r} 
~=~ 
\frac{1}{2} 
 \Bigg(-\(\frac{n}{2}-j\) \(\frac{n}{2}+j\) 
 ~+~ 
 (j-2r) (n-2k+1) \Bigg).
$$
	
\begin{rmk}
Note that  $\frac{j}{j-r} \binom{j-r}{r} \in \N$ for $0 \leq r \leq j/2$ 
and for even positive integers $k, n$ with $k > n+1$, we have
\begin{equation*}
\frac{nk}{2} - \frac{n(n+1)}{4} -\frac{n^2}{8} 
\phantom{mm}\text{and}\phantom{mm}
\frac{nk}{2} - \frac{n(n+1)}{4} ~+~ c_{j,r}
\end{equation*}
are non-negative integers for $1 \leq j \leq n/2$ and $ 0 \leq r \leq j/2$.
\end{rmk}

\medspace

\section{Hecke eigenvalues of Ikeda lifts}\label{SIkdeaeigen}
In this section, we will find a polynomial $g_p(x) \in \Z[x]$ such that
$$
\lambda_F(p) ~=~ g_p(a_f(p))
$$
for all primes $p$. We also compute the roots of the polynomial $g_p(x)$. 
As an application of this, we will show that 
$$
\lambda_F(p) ~>~ 0
$$
for all primes $p$. Further, we derive lower and upper bounds for $\lambda_F(p)$.

\begin{lem}\label{lemgpfact}
For any prime $p$, let
$$
g_p(x) ~=~ \prod_{i=1}^{\frac{n}{2}}  
\( x + p^{k-i} + p^{k-n-1+i}\).
$$
Then we have	
$$
\lambda_F(p) ~=~ g_p(a_f(p))
$$
for all primes $p$.
\end{lem}
\begin{proof}
Let $F \in S_k(\Gamma_n)$ be an Ikeda lift 
of a normalized elliptic Hecke eigenform $f \in S_{2k-n}(\Gamma_1)$. 
From \cite[Eqs. 3, 5, 7, 9, 10]{RK}, we have
$$
\lambda_F(p) ~=~ p^{\frac{nk}{2} - \frac{n(n+1)}{4}} 
\sum_{i=0}^{n} p^{\frac{i(i-n)}{2}} 
\binom{n}{i}_p \alpha_f(p)^{i-\frac{n}{2}},
$$
where $\alpha_f(p)$ is a root of the polynomial 
$x^2-a_f(p) p^{-\frac{2k-n-1}{2}} x + 1$.
Set
$$
a_i 
~=~
p^{\frac{nk}{2} -  \frac{n(n+1)}{4}} 
p^{\frac{i(i-n)}{2}} \binom{n}{i}_p
$$
and consider the polynomial
$$
G_p(x) ~=~ \sum_{i=0}^{n} a_i x^i.
$$
We have
\begin{equation}\label{eqG_pin-i}
\begin{split}
\frac{G_p(x)}{x^{\frac{n}{2}}}  
&~=~ \sum_{i=0}^{n} a_i x^{i-\frac{n}{2}} 
~=~ a_{n/2} ~+~ \sum_{i=0}^{\frac{n}{2}-1} a_i \(x^{\frac{n}{2}-i} 
+ \frac{1}{x^{\frac{n}{2}-i}}\)  \\
&~=~ 
a_{n/2}~+~
\sum_{i=0}^{\frac{n}{2}-1} a_i p^{\frac{2k-n-1}{2}(i - \frac{n}{2}) } 
\( (p^{\frac{2k-n-1}{2}}x)^{\frac{n}{2}-i} + 
\(\frac{p^{\frac{2k-n-1}{2}}}{x}\)^{\frac{n}{2}-i} \). 
\end{split}
\end{equation}
Let $x =  p^{\frac{2k-n-1}{2}}y $. Then, we have
$$
\frac{G_p(y)}{y^{\frac{n}{2}}} 
~=~ 
a_{n/2} ~+~  \sum_{i=0}^{\frac{n}{2}-1} a_i p^{\frac{2k-n-1}{2}(i - \frac{n}{2}) }
	\( x ^{\frac{n}{2} - i} 
+ \(\frac{p^{2k-n-1}}{x}\)^{\frac{n}{2}-i}\).
$$
Note that there exists a polynomial $g_{p,i}(x) \in \Z[x]$ (see \cite{Wa}) such that
$$
x ^{i} + \(\frac{p^{2k-n-1}}{x}\)^{i} 
~=~ 
g_{p,i}\(x+\frac{p^{2k-n-1}}{x}\).
$$
Hence we get
$$
\frac{G_p(y)}{y^{\frac{n}{2}}} ~
=~ 
a_{n/2} ~+~ \sum_{i=0}^{\frac{n}{2}-1} a_i p^{\frac{2k-n-1}{2}(i - \frac{n}{2}) } g_{p, \frac{n}{2}-i}
	\(x+\frac{p^{2k-n-1}}{x}\) 
~=~  
\tilde{g}_p \(x+\frac{p^{2k-n-1}}{x}\),
$$	
where 
$$
\tilde{g}_p(x) 
~=~ 
a_{n/2} ~+~
\sum_{i=0}^{\frac{n}{2}-1} 
a_i p^{\frac{2k-n-1}{2}(i - \frac{n}{2}) } 
g_{p, \frac{n}{2}-i}(x).
$$
Note that $\tilde{g}_p(x)$ is a monic polynomial in $\Z[x]$ and
\begin{equation}
\lambda_F(p)	
~=~ 
\frac{G_p(\alpha_f(p))}{\alpha_f(p)^{n/2}}
~=~ 
\tilde{g}_p \Bigg(\alpha_f(p) p^{\frac{2k-n-1}{2}} 
+ \frac{p^{2k-n-1}}{\alpha_f(p) p^{\frac{2k-n-1}{2}}} \Bigg)
~=~
\tilde{g}_p(a_f(p)).
\end{equation}
We will now show that $\tilde{g}_p(x) = g_p(x)$.
By using $q$-binomial theorem (see \thmref{thmqbinom}), 
we get
\begin{equation}
\begin{split}
\sum_{i=0}^{n} p^{\frac{i(i-n)}{2}} \binom{n}{i}_p x^i
~=~
\sum_{i=0}^{n} \binom{n}{i}_p   p^{\frac{i(i-1)}{2}} 
\( p^{\frac{1-n}{2}} x\)^i  
~=~ 
\prod_{j=0}^{n-1} \(1+p^j p^{\frac{1-n}{2}} x\).
\end{split}
\end{equation}
Hence we get
\begin{equation}\label{eqGpFact}
G_p(x) 
~=~ 
p^{\frac{nk}{2} - \frac{n(n+1)}{4}} \prod_{j=0}^{n-1} 
\(1+p^j p^{\frac{1-n}{2}} x\).
\end{equation}
In order to find factorization of $\tilde{g}_p(x)$, we proceed as follows.
We have
\begin{equation*}
\begin{split}
\tilde{g}_p\(x+\frac{p^{2k-n-1}}{x}\) 
&~=~
 \frac{G_p(y)}{y^{\frac{n}{2}}} 
~=~ 
\frac{G_p(p^{-\frac{(2k-n-1)}{2}}x)}{\(p^{\frac{-(2k-n-1)}{2}} x\)^{\frac{n}{2}}}\\
&~=~
 \(p^{-\frac{(2k-n-1)}{2}} x\)^{-\frac{n}{2}} p^{\frac{nk}{2}  - \frac{n(n+1)}{4}} \prod_{j=0}^{n-1} 
			\(1+p^j p^{\frac{1-n}{2}} p^{-\frac{(2k-n-1)}{2}} x\) \\
&~=~ p^{nk - \frac{n(n+1)}{2}} x^{-\frac{n}{2}} \prod_{j=0}^{n-1} (1+p^{j-k+1}x) \\
&~=~ x^{-\frac{n}{2}} \prod_{j=0}^{n-1} (p^{k-j-1} +x).
\end{split} 
\end{equation*}
Let $\gamma_{p, 1}, \gamma_{p, 2}, \cdots, \gamma_{p, n/2}$ 
be the roots of the polynomial $\tilde{g}_p(x)$. 
Then we have 
\begin{equation*}
\begin{split}
\prod_{j=0}^{n-1} (x + p^{k-j-1} ) 
&~=~x^{\frac{n}{2}} \prod_{i=1}^{\frac{n}{2}} \(x+\frac{p^{2k-n-1}}{x} -\gamma_{p,i}\)\\
&~=~ \prod_{i=1}^{\frac{n}{2}} \(x^2- \gamma_{p, i} x + p^{2k-n-1}\).
\end{split}
\end{equation*}
Let $\delta_{i1}$ and $\delta_{i2}$ be the roots of the polynomial
$x^2- \gamma_{p, i} x + p^{2k-n-1}$. Then we get
$$
\prod_{j=0}^{n-1} (x + p^{k-j-1}) 
~=~ \prod_{i=1}^{\frac{n}{2}} (x-\delta_{i1}) (x-\delta_{i2}).
$$
Since $\delta_{i1}\delta_{i2} = p^{2k-n-1}$, we must have 
$\{\delta_{i1}, \delta_{i2}\} = \{-p^{k-i}, -p^{k-n-1+i}\}$ 
(up to some ordering of indices). 
Thus we deduce that
$$
\gamma_{p, i} ~=~ \delta_{i1} + \delta_{i2} ~=~  -p^{k-i} - p^{k-n-1+i}.
$$
Thus we have
$$
\tilde{g}_p(x) 
~=~ \prod_{i=1}^{\frac{n}{2}}\(x - \gamma_{p,i}\) 
~=~\prod_{i=1}^{\frac{n}{2}}\(x + p^{k-i} + p^{k-n-1+i}\) 
~=~ {g}_p(x).
$$
This completes the proof of \lemref{lemgpfact}.
\end{proof}	
	
\medspace	
	
\subsection{Proof of \thmref{thm4}}
Let $F \in S_k(\Gamma_n)$ be an Ikeda lift of a 
normalized elliptic Hecke eigenform $f \in S_{2k-n}(\Gamma_1)$.
Also let $G_p(x)$ and $g_p(x)$ be as before and
$$
\tilde{G}_p(x) ~=~ \frac{G_p(x)}{x^{n/2}}.
$$
Note that 
$$
\lambda_F(p) ~=~ \tilde{G}_p(\alpha_f(p)).
$$
By Deligne's bound, we can write
$$
a_f(p) ~=~ 2p^{\frac{2k-n-1}{2}} \cos\theta_p,
~~ \theta_p \in [0, \pi]. 
$$
Then, we have $\alpha_f(p) = e^{i \theta_p}$.
To prove $\lambda_F(p) > 0$ for all primes $p$, 
it is sufficient to show that 
$$
\tilde{G}_p(z) ~>~ 0 ~~\text{ for all } z \in \mathbb{S},
$$ 
where $\mathbb{S} = \{z \in \C : |z| = 1\}$ denotes the unit circle in $\C$.
From \eqref{eqGpFact}, we have 
$$
\tilde{G}_p(z) ~\neq~ 0 ~~\text{ for all } z \in \mathbb{S}.
$$ 
From \eqref{eqG_pin-i}, note that $\tilde{G}_p(1/z) = \tilde{G}_p(z)$
and hence
$$
\tilde{G}_p(z) ~\in~ \R ~~\text{ for all } z \in \mathbb{S}.
$$  
We define a function $H$ on $\R$ by
$$
H(t) ~=~ \tilde{G}_p(e^{it}),~~ t \in \R.
$$ 
Then observe that $H$ is a real valued non-vanishing 
continuous function on $\R$
and 
$$
H(0) ~=~ \tilde{G}_p(1) ~=~ G_p(1) > 0. 
$$
This implies $H(t) > 0$ for all $t \in \R$. 
Hence we deduce that $\lambda_F(p) = \tilde{G}_p(e^{i \theta_p}) > 0$ 
for all primes~$p$. Further, we have
$$
\lambda_F(p) 
~=~ |\tilde{G}_p(e^{i \theta_p})|
~=~ |G_p(e^{i \theta_p})|
~=~ p^{\frac{nk}{2} - \frac{n(n+1)}{4}} \prod_{j=0}^{n-1} 
\left|1+p^{j+\frac{1-n}{2}} e^{i \theta_p}\right|.
$$
Note that
$$
\left|1- p^{j+\frac{1-n}{2}}\right| 
~\leq ~
\left|1+p^{j+\frac{1-n}{2}} e^{i \theta_p}\right|
~\leq~ 
1 + p^{j+\frac{1-n}{2}}.
$$
Thus we have
\begin{equation*}
\begin{split}
\lambda_F(p) 
&~\geq~
p^{\frac{nk}{2} - \frac{n(n+1)}{4}} \prod_{j=0}^{n-1} 
\left|1-p^{j+\frac{1-n}{2}}\right| \\
& ~=~
p^{\frac{nk}{2} - \frac{n(n+1)}{4}} 
\prod_{j=0}^{\frac{n}{2}-1} \(1-p^{j+\frac{1-n}{2}}\)
 \prod_{j=\frac{n}{2}}^{n-1} \(p^{j+\frac{1-n}{2}}-1\)\\
&~=~
p^{\frac{nk}{2} - \frac{n(n+1)}{4}} 
\prod_{j=0}^{\frac{n}{2}-1} \(1-\frac{1}{p^{\frac{n-1}{2}-j}}\)
 \prod_{j=\frac{n}{2}}^{n-1} p^{j+\frac{1-n}{2}} \(1-\frac{1}{p^{j-\frac{n-1}{2}}}\) \\
 &~=~
p^{\frac{nk}{2} - \frac{n(n+1)}{4} + \frac{n^2}{8}} 
\prod_{j=0}^{\frac{n}{2}-1} \(1-\frac{1}{p^{\frac{n-1}{2}-j}}\)^2. 
\end{split}
\end{equation*}
Arguing in a similar way, we get
$$
\lambda_F(p) ~\leq~ 
p^{\frac{nk}{2} - \frac{n(n+1)}{4} + \frac{n^2}{8}} 
\prod_{j=0}^{\frac{n}{2}-1} \(1+\frac{1}{p^{\frac{n-1}{2}-j}}\)^2. 
$$
This completes the proof of \thmref{thm4}.
\qed

\medspace

\section*{Acknowledgments}
The first author would like to thank DAE 
number theory plan project and SPARC project 445.
The second author would like to acknowledge the Institute of Mathematical Sciences (IMSc), India and 
Queen's University, Canada for providing excellent atmosphere to work.

\medspace

\end{document}